\documentclass[12pt,a4paper,english,reqno]{amsart}
\usepackage[a4paper,footskip=1.5em]{geometry}
\usepackage{amsmath,amssymb,amsthm,mathtools,comment}
\usepackage[mathscr]{euscript}
\usepackage[usenames,dvipsnames]{color}
\usepackage{adjustbox,tikz,calc,graphics,babel,standalone}
\usepackage{subcaption}
\usepackage{csquotes,enumerate,enumitem,verbatim}
\usepackage[final]{microtype}
\usepackage[numbers]{natbib}
\usetikzlibrary{shapes.misc,calc,intersections,patterns,decorations.pathreplacing}
\usepackage{hyperref}
\hypersetup{colorlinks=true,linkcolor=blue,citecolor=blue,pdfpagemode=UseNone,pdfstartview={XYZ null null 1.00}}
\usepackage{cmtiup}

\theoremstyle{plain}
\newtheorem*{theorem*}{Theorem}
\newtheorem{theorem}{Theorem}[section]
\newtheorem{lemma}[theorem]{Lemma}
\newtheorem{claim}[theorem]{Claim}
\newtheorem{proposition}[theorem]{Proposition}
\newtheorem*{claim*}{Claim}
\newtheorem{corollary}[theorem]{Corollary}

\newtheorem{problem}[theorem]{Problem}
\newtheorem{question}[theorem]{Question}
\theoremstyle{remark}

\DeclareMathOperator{\Aut}{Aut}
\DeclareMathOperator\Bi{Bin}

\newcommand{\A}{\mathcal{A}}
\newcommand{\B}{\mathcal{B}}
\newcommand{\F}{\mathcal{F}}

\newcommand{\PN}{\mathcal{P}_n}
\newcommand{\N}{\mathbb{N}}
\newcommand{\E}{\mathbb{E}}
\newcommand{\NK}{[n]^{(k)}}
\newcommand{\Prob}{\mathbb{P}}

\let\emptyset\varnothing
\let\eps\varepsilon
\let\originalleft\left
\let\originalright\right
\renewcommand{\left}{\mathopen{}\mathclose\bgroup\originalleft}
\renewcommand{\right}{\aftergroup\egroup\originalright}
\makeatletter
\def\imod#1{\allowbreak\mkern10mu({\operator@font mod}\,\,#1)}
\makeatother

\begin{document}

\title{On symmetric intersecting families}

\author{David Ellis}
\address{School of Mathematics, University of Bristol, Fry Building, Woodland Road, Bristol BS8\thinspace1UG, UK}
\email{david.ellis@bristol.ac.uk}

\author{Gil Kalai}
\address{Einstein Institute of Mathematics, The Hebrew University of Jerusalem, Edmond
J. Safra Campus, Jerusalem 91904, Israel}
\email{kalai@math.huji.ac.il}

\author{Bhargav Narayanan}
\address{Department of Mathematics, Rutgers University, Piscataway NJ 08854, USA}
\email{narayanan@math.rutgers.edu}

\date{21st July 2018}


\begin{abstract}
We make some progress on a question of Babai from the 1970s, namely:  for $n, k\in \N$ with $k \le n/2$, what is the largest possible cardinality $s(n,k)$ of an intersecting family of $k$-element subsets of $\{1,2,\ldots,n\}$ admitting a transitive group of automorphisms? We give upper and lower bounds for $s(n,k)$, and show in particular that $s(n,k) = o (\binom{n-1}{k-1})$ as $n \to \infty$ if and only if $k = n/2 - \omega(n)(n/\log n)$ for some function $\omega(\cdot)$ that increases without bound, thereby determining the threshold at which `symmetric' intersecting families are negligibly small compared to the maximum-sized intersecting families. We also exhibit connections to some basic questions in group theory and additive number theory, and pose a number of problems.
\end{abstract}
\maketitle

\section{Introduction}
A family of sets is said to be \emph{intersecting} if any two sets in the family have nonempty intersection, and \emph{uniform} if all the sets in the family have the same size. In this paper, we study uniform intersecting families. The most well-known result about such families is the Erd\H{o}s--Ko--Rado theorem~\citep{EKR}.
\begin{theorem}
\label{thm:ekr}
Let $n,k \in \mathbb{N}$ with $k \le n/2$. If $\A$ is an intersecting family of $k$-element subsets of $\{1,2,\dots,n\}$, then $|\A| \le \binom{n-1}{k-1}$. Furthermore, if $k < n/2$, then equality holds if and only if $\A$ is a star, meaning that $\A$ consists of all the $k$-element subsets of $\{1,2,\ldots,n\}$ that contain some fixed element $i \in \{1,2,\dots,n\}$.
\end{theorem}

Over the last fifty years, many variants of this theorem have been obtained. A common feature of many of these variants is that the extremal families are highly asymmetric; this is the case, for example, in the Erd\H{o}s--Ko--Rado theorem itself, in the Hilton--Milner theorem~\citep{HM}, and in Frankl's generalisation~\citep{FranklHM} of these results. It is therefore natural to ask what happens to the maximum possible size of a uniform intersecting family when one imposes some kind of symmetry requirement on the family.

In the 1970s, Babai posed the problem of determining the maximum possible size of a uniform intersecting family with transitive automorphism group; this is a very natural symmetry requirement to impose. In this paper, we make some progress on Babai's problem.

Let us first give our definitions in full, and fix some notation. For a positive integer $n \in \N$, we denote the set $\{1, 2,\dots, n\}$ by $[n]$. We write $S_n$ for the symmetric group on $[n]$ and $\PN$ for the power-set of $[n]$. For a permutation $\sigma \in S_n$ and a set $x \subset [n]$, we write $\sigma(x)$ for the image of $x$ under $\sigma$, and if $\A \subset \PN$, we write $\sigma(\A) = \{\sigma(x):x \in \A\}$. We define the \emph{automorphism group} of a family $\A \subset \PN$ by
\[\Aut(\A) = \{\sigma \in S_n:\sigma(\A) = \A\}.\]
We say that $\A \subset \PN$ is \emph{symmetric} if $\Aut(\A)$ is a transitive subgroup of $S_n$, i.e., if for all $i, j \in [n]$, there exists a permutation $\sigma \in \Aut(\A)$ such that $\sigma(i) = j$.

For a pair of integers $n,k \in \N$ with $k \le n$, let $\NK$ denote the family of all $k$-element subsets of $[n]$, and let
\[s(n,k) = \max\{|\A|:\A \subset \NK \text{ such that } \A \text{ is symmetric and intersecting}\}.\]
Of course, if $k > n/2$, then $[n]^{(k)}$ itself is a symmetric intersecting family, so $s(n,k) = \binom{n}{k}$; in studying $s(n,k)$, we may therefore restrict our attention to the case where $k \le n/2$.

With these definitions in place, we may state the aforementioned question of Babai more precisely.
\begin{problem}\label{bab}
For $n,k \in \N$ with $k \le n/2$, determine $s(n,k)$.
\end{problem}
We remark that, in the non-uniform setting (where one studies families of sets not all of the same size, i.e., subsets of $\PN$), several authors have obtained results on the maximum size of symmetric families that are intersecting (or satisfy some stronger intersection requirement); see for example the results of Frankl~\citep{frankl-1}, Cameron, Frankl and Kantor~\citep{frankl-2}, and the more recent results of the first and third authors~\citep{en}. Relatively little seems to be known in the uniform setting, however. 

As a first step towards Problem~\ref{bab}, we focus on determining when a symmetric uniform intersecting family must be \emph{significantly smaller} than the extremal families (of the same uniformity) in the Erd\H{o}s--Ko--Rado theorem. A more precise formulation of this question is as follows.
\begin{question}\label{p:main}
For which $k = k(n) \le n/2$ is $s(n,k) = o(\binom{n-1}{k-1})$ as $n \to \infty$?
\end{question}

Utilising a well-known sharp threshold result of Friedgut and the second author~\citep{FK}, we prove the following.
\begin{theorem}
\label{thm:main}
There exists a universal constant $c>0$ such that
\[s(n,k) \le 2\exp\left(-\frac{c(n-2k)\log n}{n} \right) \binom{n}{k}\]
for any $n, k \in \N$ with $k \le n/2$.
\end{theorem}

We also give a construction showing that Theorem~\ref{thm:main} is sharp up to the value of $c$ in the regime where $k/n$ is bounded away from zero. This construction, in conjunction with Theorem~\ref{thm:main}, provides a complete answer to Question~\ref{p:main}.

\begin{theorem}\label{prop:tight}
If $k = k(n) \le n/2$, then as $n \to \infty$, $s(n,k) = o(\binom{n-1}{k-1})$ if and only if \[k = \frac n2 - \omega(n)\left(\frac{n}{\log n}\right)\] for some function $\omega(\cdot)$ that increases without bound.
\end{theorem}

While Question~\ref{p:main} is the most basic question in the regime where the uniformity $k$ is large compared to the size $n$ of the ground set, the most basic question in the regime where $k$ is small compared to $n$ concerns the \emph{existence} of symmetric intersecting families. Note that if $s(n,k) >0$, then $s(n,l) >0$ for all $l > k$; indeed, if $\A \subset [n]^{(k)}$ is nonempty, symmetric and intersecting, then so is $\{y \in [n]^{(l)}: x \subset y \text{ for some }x \in \A\}$. With this in mind, for each $n \in \mathbb{N}$, we define
\[g(n) = \min\{k \in \mathbb{N}: s(n,k)>0\}.\]

It turns out that problem of determining the function $g(\cdot)$ is intimately connected to some longstanding open problems in group theory and additive number theory. It is not hard to show, as we shall see, that $g(n) \ge \sqrt{n}$ for all $n \in \N$. It is then natural to ask whether this bound is asymptotically tight, and this prompts us to raise the following question.
\begin{question}
\label{qn:threshold}
Is it true that $g(n) = (1+o(1))\sqrt{n}$ for all $n \in \mathbb{N}$?
\end{question}
While we are able to show that the estimate in Question~\ref{qn:threshold} holds along various (arithmetically special) sequences of positive integers, we are unfortunately unable to settle this question entirely. 

The remainder of this paper is organised as follows. We give the proof of Theorem~\ref{thm:main} in Section~\ref{s:proof}. In Section~\ref{s:cons}, we describe a combinatorial approach to constructing large symmetric intersecting families in the regime where $k$ is comparable to $n$, and deduce Theorem~\ref{prop:tight} as a consequence. In Section~\ref{s:smallk}, we turn to the regime where $k$ is comparable to $\sqrt{n}$, and describe various algebraic constructions of symmetric intersecting families in this regime. We finally conclude in Section~\ref{s:conc} with a discussion of some open problems and related work.

\section{Upper bounds}\label{s:proof}
We first describe briefly the notions and tools we will need for the proof of Theorem~\ref{thm:main}. In what follows, all logarithms are to the base $e$.

We begin with the following simple observation which may be found in~\citep{frankl-2}, for example; we include a proof for completeness. 

\begin{proposition}\label{rootn}
For all $n,k \in \N$ with $1 < k \le \sqrt{n}$, we have $s(n,k) = 0$.
\end{proposition}
\begin{proof}
The proposition follows from a simple averaging argument. Indeed, let $k \le \sqrt{n}$, suppose for a contradiction that $\A \subset \NK$ is a nonempty, symmetric intersecting family, and let $x \in \A$. If we choose $\sigma \in \Aut(\A)$ uniformly at random, then since $\Aut(\A)$ is transitive, we have
\[\E[ |x \cap \sigma(x)|] = \frac{k^2}{n} \le 1,\]
where the first equality above depends on the fact that
\[|\sigma \in \Aut(\A) : \sigma(i) = j| = |\sigma \in \Aut(\A) : \sigma(i) = k|\]
for all $i, j, k \in [n]$. Since $|x \cap \text{Id}(x)| = k > 1$, there must exist a permutation $\sigma \in \Aut(\A)$ such that $x \cap \sigma(x) = \emptyset$, contradicting the fact that $\A$ is intersecting.
\end{proof}

For $0 \le p \le 1$, we write $\mu_p$ for the \emph{$p$-biased measure} on $\PN$, defined by
\[\mu_p(\{x\}) = p^{|x|}(1-p)^{n-|x|} \]
for each $x \subset [n]$, and by
$$\mu_p(\mathcal{F}) = \sum_{x \in \mathcal{F}}\mu_p(\{x\})$$
for each $\mathcal{F} \subset \PN$. We say that a family $\F \subset \PN$ is {\em increasing} if it is closed under taking supersets, i.e., if $x \in \F$ and $x \subset y$, then $y \in \F$.  It is easy to see that if $\F \subset \PN$ is increasing, then $p \mapsto \mu_p(\F)$ is a monotone non-decreasing function on $[0,1]$. For a family $\F \subset \PN$, we write $\F^{\uparrow}$ for the smallest increasing family containing $\F$; in other words, $\F^{\uparrow} = \{y \subset [n]: x \subset y \text{ for some }x \in \F\}$. 

For any family $\mathcal{A} \subset [n]^{(k)}$, we write
$$\partial^+ \mathcal{A}: = \{x \in [n]^{(k+1)}:\ x \supset y \text{ for some }y \in \mathcal{A}\}$$
for the {\em upper shadow} of $\mathcal{A}$, and 
$$\partial^{+(t)}(\mathcal{A}) := \{x \in [n]^{(k+t)}:\ x \supset y \text{ for some }y \in \mathcal{A}\} = \partial^{+}(\partial^{+(t-1)}\mathcal{A})$$
for its $t$th iterate (for each $t \in \mathbb{N}$ with $t\leq n-k$). The local LYM inequality (see e.g.\ \cite{bollobas}, \S 3) states that for any integers $1 \leq k < n$ and any family $\mathcal{A} \subset [n]^{(k)}$, we have
$$\frac{|\partial^{+} \mathcal{A}|}{{n \choose k+1}}\geq \frac{|\mathcal{A}|}{{n \choose k}}.$$
Iterating the local LYM inequality yields
\begin{equation}\label{eq:ll-it} \frac{|\partial^{+(t)} \mathcal{A}|}{{n \choose k+t}}\geq \frac{|\mathcal{A}|}{{n \choose k}}\end{equation}
for all $t \leq n-k$.

We need the following fact, which allows one to bound from above the size of a family $\F \subset \NK$ in terms of $\mu_p(\F^{\uparrow})$, where $p \approx k/n$; this was proved in a slightly different form by Friedgut~\citep{F}. We provide a proof for completeness.
\begin{lemma}
\label{lem:chernoff}
Let $n,k \in \mathbb{N}$ and suppose that $0 < p,\phi < 1$ satisfy
\[p \ge \frac{k}{n} + \frac{\sqrt{2n \log (1/\phi)}}{n}.\]
Then for any family $\F \subset \NK$, we have
\[\mu_p(\F^{\uparrow}) > (1-\phi) \frac{|\F|}{\binom{n}{k}}.\]
\end{lemma}
\begin{proof}
Let $\delta = |\F|/\binom{n}{k}$ and let $X$ be a binomial random variable with distribution $\Bi(n,p)$. First, for each $l\ge k$, (\ref{eq:ll-it}) implies that 
\[\frac{|\F^{\uparrow} \cap [n]^{(l)} |}{ \binom{n}{l}} \ge \frac{|\F|}{\binom{n}{k}} = \delta.\]
Now, for any $\eta > 0 $, it follows from a standard Chernoff bound that
\[\Prob(X < (1-\eta)np) < \exp(-\eta^2np/2). \]
Hence, 
\begin{align*}
\mu_{p}\left(\F^{\uparrow}\right) & \ge\sum_{l=k}^{n}p^{l}\left(1-p\right)^{n-l}\binom{n}{l}\delta\\
&= \Prob(X \ge k) \delta\\ 
&> (1-\phi)\delta,
\end{align*}
where the last inequality above follows from a standard calculation.
\end{proof}

We will also require the following sharp threshold result due to Friedgut and the second author~\citep{FK}.
\begin{theorem}
\label{thm:fk}
There exists a universal constant $c_0>0$ such that the following holds for all $n \in \mathbb{N}$. Let $0 \leq p < 1$, $0 < \epsilon < 1$ and let $\F \subset \PN$ be a symmetric increasing family. If $\mu_p(\F)> \eps$, then $\mu_q(\F) > 1 - \eps$, where
\[q = \min\left\{1,p + c_0 \frac{\log \tfrac{1}{\epsilon}}{ \log n}\right\}.\eqno\qed\]
\end{theorem}

The idea of the proof of Theorem~\ref{thm:main} is as follows. Let $\A \subset \NK$ be a symmetric intersecting family. We first use Lemma~\ref{lem:chernoff} to bound $|\A|/\binom{n}{k}$ from above in terms of $\mu_{p}(\A^{\uparrow})$, where $p \approx k/n$; we then use Theorem~\ref{thm:fk}, together with the simple fact that $\mu_{1/2}(\A^{\uparrow}) \le 1/2$, to bound $\mu_{p}(\A^{\uparrow})$, and hence $|\A|$, from above.  Let us also remark that this strategy of `approximating' the uniform measure on $[n]^{(k)}$ with the $p$-biased measure $\mu_p$, where $p \approx k/n$, has proven useful on a number of different occasions in the study of uniform intersecting families; see~\citep{F, df, elk}, for example.
\begin{proof}[Proof of Theorem~\ref{thm:main}]
Let $n,k \in \mathbb{N}$ with $k \le n/2$, let $\A \subset \NK$ be a symmetric intersecting family, and set $\delta = |\A|/\binom{n}{k}$.

First, applying Lemma~\ref{lem:chernoff} with $p = k/n + \sqrt{(2 \log 2) n}/n$ and $\phi = 1/2$, we see that
\[\mu_p(\A^{\uparrow}) > \frac{\delta}{2}.\]
Next, since $\A$ is symmetric, so is $\A^{\uparrow}$. We may therefore apply Theorem~\ref{thm:fk} with $\eps = \delta/2$ to deduce that $\mu_q(\A^{\uparrow}) > 1/2$, where
\[q = \min\left\{1,p + c_0\left(\frac{\log (2/\delta)}{ \log n}\right)\right\}.\]

Since $\A^{\uparrow}$ is increasing, the function $r \mapsto \mu_r(\mathcal{A}^{\uparrow})$ is monotone non-decreasing on $[0,1]$. Also, since $\A$ is intersecting, so is $\A^{\uparrow}$, and therefore $\mu_{1/2}(\A^{\uparrow}) \le 1/2$. Now, as $\mu_{1/2}(\A^{\uparrow}) \le 1/2$ and $\mu_q(\A^{\uparrow}) > 1/2$, the monotonicity of $r \mapsto \mu_r(\mathcal{A}^{\uparrow})$ implies that
\[p + c_0 \left(\frac{\log (2/\delta)}{ \log n}\right) > \frac{1}{2}.\]
Rearranging this inequality, we see that
\[\delta < 2\exp\left( -\frac{(1-2p)\log n}{2c_0}\right) \le  2\exp\left(-\frac{c(n-2k)\log n}{n} \right),\]
where $c>0$ is a universal constant; this proves the theorem.
\end{proof}

\section{Lower bounds for large $k$}\label{s:cons}
In this section, we give a construction showing that Theorem~\ref{thm:main} is sharp up to the value of the constant $c$ in the exponent for many choices of $k = k(n)$. 

Given $n ,k \in\N$ with $k \le n$, we identify $[n]$ with $\mathbb{Z}_n$, we identify a subset $S \subset \mathbb{Z}_n$ with its characteristic vector $\chi_S \in \{0,1\}^{\mathbb{Z}_n}$, and we define $\F(n,k)$ to be the family of all $k$-element subsets $S \subset \mathbb{Z}_n$ such that the longest run of consecutive ones in $\chi_S$ is longer than the longest run of consecutive zeros in $\chi_S$. Slightly less formally, we take $\F(n,k)$ to consist of all the cyclic strings of $n$ zeros and ones which contain exactly $k$ ones and in which the longest run of consecutive ones is longer than the longest run of consecutive zeros. 

It is clear that $\F(n,k)$ is symmetric. It is also easy to check that $\F(n,k)$ is intersecting. Indeed, given $S,T \in \F(n,k)$, suppose without loss of generality that the longest run of consecutive ones in $S$ is at least as long as that in $T$. Choose a run of consecutive ones in $S$ of the maximum length; these cannot be all zeros in $T$ because otherwise $T$ would have a longer run of consecutive ones than $S$. Therefore, $S \cap T \neq \emptyset$.

We note that the non-uniform case of this construction, i.e., the family of all cyclic strings of $n$ zeros and ones in which the longest run of consecutive ones is longer than the longest run of consecutive zeros, shows that the Kahn--Kalai--Linial theorem~\citep{kkl} cannot be improved by more than a constant factor for intersecting families; see~\citep{mathoverflow} for more details.

An analysis of $\mathcal{F}(n,k)$ yields the following. \begin{lemma}
\label{l:lbgen}
There exists a universal constant ${\hat C}>0$ such that if $k = k(n) \in \mathbb{N}$ is such that $\sqrt{n} \log n \leq k \leq n/2$ for all $n \in \mathbb{N}$, then
\begin{align}\label{eq:fnk}
|\F(n,k)| & \ge n \cdot \exp\left(-(1+\hat{C}/\log n)\left(\frac{\log n - \log k}{\log n - \log (n-k)}\right) \log n \right) \binom{n}{k}\\
& = \exp\left(-(1+\hat{C}/\log n)\left(\frac{\log n - \log k}{\log n - \log (n-k)}\right) \log n + \log n \right) \binom{n}{k} \nonumber
\end{align}
\end{lemma}

A comparison of Theorem \ref{thm:main} and (\ref{eq:fnk}) shows that Theorem~\ref{thm:main} is sharp up to the constant factor in the exponent when $k/n$ is bounded away from zero and $1/2-k/n = \Omega(1/\log n)$.

Lemma \ref{l:lbgen} implies the following, in the case where $k/n$ is close to $1/2$.

\begin{lemma}\label{l:lb}
	For each $C > 0$, there exists $c > 0$ such that for any $n,k\in\N$ with $\tfrac{1}{2}(1 - \tfrac{C}{\log n}) \le \tfrac{k}{n} \le \tfrac{1}{2}$, we have
	\[|\F(n,k)| \ge c \binom{n-1}{k-1}.\]
\end{lemma}

To prove Lemma~\ref{l:lbgen}, we need the following.

\begin{lemma}
\label{lem:counting-lemma}
Let $k < n$. The number of cyclic strings of $n$ zeros and ones with exactly $k$ ones and a run of consecutive zeros of length at least $l$ is at most $\tfrac{1}{4} \binom{n}{k}$, provided
\[l \ge \frac{\log n + 2 \log 2}{\log n - \log (n-k)}.\]
\end{lemma}
\begin{proof}
The number of such strings is at most $n \binom{n-l}{k}$, since (possibly overcounting) there are $n$ choices for the position of the run of $l$ consecutive zeros, and then $\binom{n-l}{k}$ choices for the positions of the ones. We have
\[\frac{n \binom{n-l}{k}}{\binom{n}{k}} = \frac{n(n-k)(n-k-1)\dots(n-k-l+1)}{n(n-1)\dots(n-l+1)} \le n\left(\frac{n-k}{n}\right)^l \le \frac{1}{4}\]
provided $l \ge (\log n + 2 \log 2)/(\log n - \log (n-k))$, as required.
\end{proof}

We now make the following straightforward claim.
\begin{claim}
\label{claim:flip}
If $1 \leq l \leq k \leq n/2$, then the number of cyclic strings of length $n$ with $k$ ones and a run of consecutive ones of length at least $l$ is at most the number of cyclic strings of length $n$ with $k$ ones and a run of consecutive zeros of length at least $l$.
\end{claim}
\begin{proof}[Proof of Claim \ref{claim:flip}.]
Let $k \leq n/2$ and let $\mathcal{A} \subset [n]^{(k)}$. Applying (\ref{eq:ll-it}) with $t = n-2k$, and using the fact that ${n\choose k} = {n \choose n-k}$, yields
\begin{equation}\label{eq:comp}\frac{|\partial^{+(n-2k)} \mathcal{A}|}{{n \choose k}} = \frac{|\partial^{+(n-2k)} \mathcal{A}|}{{n \choose n-k}} \geq \frac{|\mathcal{A}|}{{n \choose k}}.\end{equation}
Now let $1\leq l \leq k \leq n/2$, let $\mathcal{A}$ be the family of cyclic strings of length $n$ with $k$ ones and a run of consecutive ones of length at least $l$, and let $\mathcal{B}$ be the family of cyclic strings of length $n$ with $n-k$ ones and a run of consecutive ones of length at least $l$. Clearly, we have $\mathcal{B} = \partial^{+(n-2k)}\mathcal{A}$, and therefore by (\ref{eq:comp}), we have $|\mathcal{A}| \leq |\mathcal{B}|$. But, by flipping zeros and ones, it is clear that $|\mathcal{B}|$ is precisely the number of cyclic strings of length $n$ with $k$ ones and a run of consecutive zeros of length at least $l$, proving the claim.
\end{proof}

The following is immediate from Lemma \ref{lem:counting-lemma} and Claim \ref{claim:flip}.
\begin{corollary}
\label{corr:none}
Let $k \le n/2$. The number of cyclic strings of length $n$ with $k$ ones and no run of $l$ consecutive zeros or ones is at least $\tfrac{1}{2} \binom{n}{k}$, provided
\[l \ge \frac{\log n + 2 \log 2}{\log n - \log (n-k)}. \eqno\qed\]
\end{corollary}

We are now ready to prove Lemma~\ref{l:lbgen}.
\begin{proof}[Proof of Lemma~\ref{l:lbgen}]
Choose $l_0 \in \mathbb{N}$ such that 
\begin{equation} \label{eq:l-condition} l_0-1 \ge \frac{\log (n-l_0-2) + 2 \log 2}{\log (n-l_0-2) - \log (n-k-2)}. \end{equation}
Observe that $\F(n,k)$ contains all cyclic strings of length $n$ with $k$ ones, precisely one run of $l_0$ consecutive ones, all other runs of consecutive ones having length at most $l_0-2$, and no run of $l_0$ consecutive zeros. We claim that if $l_0 < n/2$, then the number of such strings is at least
\begin{equation}\label{eq:comb-bound}\frac{n}{2} \binom{n-l_0-2}{k - l_0}.\end{equation}
Indeed, there are $n$ choices for the position of the run of $l_0$ consecutive ones, and there must be a zero on each side of this run of ones. Now, there are at least $\tfrac{1}{2} \binom{n-l_0-2}{k - l_0}$ choices for the remainder of the cyclic string (by Corollary~\ref{corr:none}), since if we take a cyclic string of length $n-l_0-2$ which contains no run of $l_0-1$ consecutive ones or zeros, and then insert (at some point) a run of $l_0$ consecutive ones with a zero on either side into this string, then the resulting string has the desired property provided $l_0 < n/2$.

It is easily checked that if $\sqrt{n} \log n \leq k \leq n/2$, then we may choose $l_0 \in \mathbb{N}$ satisfying~\eqref{eq:l-condition} such that
\begin{equation}\label{eq:l0bound} l_0 = (1+O(1/\log n)) \frac{\log n}{\log n - \log (n-k)}.\end{equation}
Indeed, choose 
$$l_0 = (1+\epsilon) \frac{\log n}{\log n - \log (n-k)},$$
where $\epsilon= O(1/\log n)$ is to be chosen later. Then, using the fact that $\log n - \log (n-k) = \Omega(k/n)$, we have $l_0 = O((n \log n)/k) = O(\sqrt{n})$, and therefore
\begin{align}\label{eq:bound-frac-1}\frac{\log (n-l_0-2)-\log(n-k-2)}{\log n - \log(n-k)} & = 1-\frac{\log(1+\tfrac{l_0+2}{n-l_0-2}) -\log(1+\tfrac{2}{n-k-2})}{\log n - \log(n-k)} \nonumber\\
& = 1-O(\tfrac{l_0}{k}) \nonumber\\
& = 1-O((n \log n)/k^2) \nonumber\\
& = 1-O(1/\log n).\end{align}
Provided $\epsilon \geq C/\log n$ for some absolute constant $C$, we have
\begin{equation}\label{eq:bound-frac-2}l_0 - 1 \geq (1+\tfrac{C}{2\log n})\frac{\log n}{\log n - \log (n-k)}.\end{equation}
Finally, we clearly have
\begin{equation}\label{eq:bound-frac-3}\frac{\log n}{\log (n-l_0-2)+2\log 2} = 1-O(1/\log n).\end{equation}
Putting (\ref{eq:bound-frac-1}), (\ref{eq:bound-frac-2}) and (\ref{eq:bound-frac-3}) together, we obtain
\begin{align*} l_0 - 1 & \geq (1+\tfrac{C}{2\log n})\frac{\log n}{\log n - \log (n-k)}\\
&\geq (1+\tfrac{C}{2\log n})(1-O(1/\log n)) \frac{\log (n-l_0-2)+2\log 2}{\log (n-l_0-2)-\log(n-k-2)},\end{align*}
yielding (\ref{eq:l-condition}) provided $C$ is sufficiently large.

Provided $n$ is at least an absolute constant (which we may assume), we have $l_0 < n/2$, and therefore, using (\ref{eq:comb-bound}), we have
\begin{align*} |\F(n,k)| & \ge \frac{n}{2} \binom{n-l_0-2}{k - l_0} \ge \frac{n}{2} \binom{n-l_0-2}{k - l_0-2}\\
& \ge \frac{n}{2} \left(\frac{k-l_0-2}{n-l_0-2}\right)^{l_0+2} \binom{n}{k}\\
&\ge \Omega(1) \cdot n \cdot \left(\tfrac{k}{n}\right)^{l_0+2}\binom{n}{k}\\
& \ge n \cdot \exp\left(-(1+O(1/\log n))\left(\frac{\log n - \log k}{\log n - \log (n-k)}\right) \log n \right) \binom{n}{k},
\end{align*}
proving the lemma.
\end{proof}

We now deduce Lemma~\ref{l:lb} from Lemma \ref{l:lbgen}.
\begin{proof}[Proof of Lemma~\ref{l:lb}]
Defining $\eta: = 1/2-k/n$, we have $\eta\leq C/\log n$, and
\[\frac{\log n - \log k}{\log n - \log (n-k)} = \frac{\log 2 - \log(1-2\eta)}{\log 2 - \log(1+2\eta)} = 1+O(\eta).\]
Hence, it follows  from Lemma~\ref{l:lbgen} that
\[|\F(n,k)| \ge n \cdot \exp\left(-(1+\hat{C}/\log n)(1+O(C/\log n))\log n \right) \binom{n}{k} \geq c\binom{n-1}{k-1}\]
provided $c$ is sufficiently small depending on $\hat{C}$ and $C$, as required.
\end{proof}

Theorem~\ref{prop:tight} is immediate from Theorem \ref{thm:main} and Lemma \ref{l:lb}.

\section{Lower bounds for small $k$}\label{s:smallk}
In the previous two sections, we focussed on estimating the largest possible cardinality $s(n,k)$ of a symmetric intersecting subfamily of $\NK$. We now turn our attention to estimating the smallest possible uniformity $k = g(n)$ for which there exists a nonempty, symmetric intersecting subfamily of $\NK$. To this end, we will investigate the set
\[\mathcal{S}=\{(n,k) \in \mathbb{N}^2:s(n,k) >0\}.\]
Along the way, we describe some constructions of symmetric intersecting families that are larger than $\F(n,k)$ for certain values of $n$ and $k$. 

We have already seen (in Proposition~\ref{rootn}) that $g(n) > \sqrt{n}$ for all $n \ge 2$. Let us now consider upper bounds on $g(n)$.

It is easy to check that $\F(n,k) \neq \emptyset$ provided $n \leq \lfloor k^2/4\rfloor +k$. (Consider the cyclic string $1^{\ell} (0^{\ell-1}1)^{t}0^{\ell-1}$, where $\ell = \lfloor k/2\rfloor+1$, $t = \lceil k/2\rceil-1$, and $n = \lfloor k^2/4\rfloor +k$; here, as usual, if $S$ is a string, $S^N$ denotes $S$ concatenated with itself $N$ times.) This observation implies that
\begin{equation} \label{eq:triv-bound} g(n) \leq 2\sqrt{n}\end{equation}
for all $n\in\N$. To improve~\eqref{eq:triv-bound}, we note a strong connection between the problem of determining $g(n)$ and the problem of finding a so-called {\em difference cover} in an Abelian group. If $G$ is a finite Abelian group and $S \subset G$, we say that $S$ is a {\em difference cover for $G$} if $S - S = G$, i.e., if $\{ i - j: i,j \in S\} = G$; we then define
\[h(G) = \min\{|S|: S \text{ is a difference cover for }G\}.\]
Note that if $S \subset G$, then $S$ is a difference cover for $G$ if and only if the family of all the translates of $S$ is an intersecting family of subsets of $G$. This observation yields the following.
\begin{lemma}
\label{l:diff-cov}
For all $n \in \mathbb{N}$, we have $g(n) \leq h(\mathbb{Z}_n)$, with equality holding in the case where $n$ is prime.
\end{lemma}
\begin{proof}
Let $h=h(\mathbb{Z}_n)$ and write $\mathbb{Z}_n^{(h)}$ for the family of $h$-element subsets of $\mathbb{Z}_n$. By definition, there exists $S \in \mathbb{Z}_n^{(h)}$ such that $S - S = \mathbb{Z}_n$. Let $\A = \{S + j: j \in \mathbb{Z}_n\} \subset \mathbb{Z}_n^{(h)}$ denote the family of all the translates of $S$. Then $\A$ is clearly symmetric and intersecting. Hence, $g(n) \leq h$, proving the first part of the claim.

Now suppose that $n$ is prime, and let $g(n)=k$. Let $\mathcal{A} \subset [n]^{(k)}$ be a nonempty, symmetric intersecting family. Since $\Aut(\mathcal{A}) \leq S_{n}$ is transitive, the orbit-stabilizer theorem implies that $n$ divides $|\Aut(\A)|$, and therefore by Cauchy's theorem, $\Aut(\mathcal{A})$ has a cyclic subgroup $H$ of order $n$. Let $\sigma \in S_{n}$ be a generator of $H$; then $\sigma$ is an $n$-cycle, and by relabelling the ground set $[n]$ if necessary, we may assume that $\sigma = (1\ 2\  \dots \ n)$ (in the standard cycle notation). Fix $x \in \mathcal{A}$ and note that $\mathcal{B}=\{x,\sigma(x),\dots,\sigma^{n-1}(x)\}$ is also a nonempty, symmetric intersecting family as $H \leq \Aut(\mathcal{B})$. Clearly, $\mathcal{B}$ consists of all the cyclic translates, modulo $n$, of $x$. If we regard $x$ as a subset of $\mathbb{Z}_n$, then since $\mathcal{B}$ is intersecting, we have $x - x = \mathbb{Z}_n$, i.e., $x$ is a difference cover for $\mathbb{Z}_n$. Hence, $h(\mathbb{Z}_n) \leq k$ and it follows that $h(\mathbb{Z}_n) = g(n)$ when $n$ is prime, as required.
\end{proof}

We now describe how existing constructions of difference covers lead to an improvement of~\eqref{eq:triv-bound}. We say that $S \subset \mathbb{Z}$ is a {\em difference cover for $n$} if $[n] \subset S-S $. For each $n \in \mathbb{N}$, let $\pi_n\colon \mathbb{Z} \to \mathbb{Z}_n$ denote the natural projection modulo $n$ defined by $\pi_n(i) = i \imod n$ for all $i \in \mathbb{Z}$. Note that if $S \subset \mathbb{Z}$ is a difference cover for $\lfloor n/2 \rfloor$, then $\pi_n(S)$ is a difference cover for $\mathbb{Z}_n$. Building on work of R\'edei and R\'enyi~\citep{rr} and of Leech~\citep{leech}, Golay~\citep{golay} proved that for any $n \in \mathbb{N}$, there exists a difference cover for $n$ of size at most $\sqrt{cn}$, where $c<2.6572$ is an absolute constant. It follows that for any $n \in \mathbb{N}$, we have
\[g(n) \leq h(\mathbb{Z}_n) \leq 1.1527 \sqrt{n}.\]
Unfortunately, one cannot hope to answer Question~\ref{qn:threshold} in the affirmative purely by projecting difference covers for $\lfloor n/2 \rfloor$ into $\mathbb{Z}_n$ and using the fact that $g(n) \leq h(\mathbb{Z}_n)$; this is a consequence of a result of R\'edei and R\'enyi~\citep{rr} which asserts that if $S \subset \mathbb{Z}$ is a difference cover for $n$, then
\[|S| \geq \sqrt{\left(2+\frac{4}{3\pi}\right)n}.\]

In view of Lemma~\ref{l:diff-cov}, we are led to the following question, which being a natural question in its own right, has also occurred independently to others; see~\citep{banakh}, for instance.
\begin{question}
\label{q:cov}
Is it true that $h(\mathbb{Z}_n) = (1+o(1))\sqrt{n}$ for all $n \in \mathbb{N}$?
\end{question}
By Lemma~\ref{l:diff-cov}, an affirmative answer to this question would imply an affirmative answer to Question~\ref{qn:threshold}. We remark that Question~\ref{q:cov} is a `covering' problem whose `packing' counterpart has received a lot of attention. If $G$ is an Abelian group and $S \subset G$, we say that $S$ is a {\em Sidon set} in $G$ if for any non-identity element $g \in G$, there exists at most one ordered pair $(s_1,s_2) \in S^2$ such that $g = s_1-s_2$. For $n \in \mathbb{N}$, let
\[\lambda(n) = \max\{|S|: S \subset \mathbb{Z}_n \text{ such that } S \text{ is a Sidon set}\}.\]
The determination of $\lambda(n)$ is a well-known open problem; see~\citep{sidon}, for example. In particular, the following remains open.
\begin{question}
\label{qn:s}
Is it true that $\lambda(n) = (1-o(1))\sqrt{n}$ for all $n \in \mathbb{N}$?
\end{question}
The constructions of Singer~\citep{singer} and Bose~\citep{bose} yield affirmative answers to Question~\ref{qn:s} when $n$ is of the form $q^2+q+1$ or $q^2-1$ respectively, where $q$ is a prime power, and a construction due to Ruzsa~\citep{ruzsa} does so when $n$ is of the form $p^2-p$, where $p$ is prime; as observed by Banakh and Gavrylkiv~\citep{banakh}, these constructions of Singer, Bose and Ruzsa yield efficient difference covers as well, so we also have affirmative answers to Questions~\ref{q:cov} and~\ref{qn:threshold} for all $n$ of the aforementioned form.

Returning to the question of determining $g(n)$, we have shown that
\begin{equation}\label{eq:containment}
 \lfloor\sqrt{n} \rfloor + 1 \le g(n) \le 1.1527 \sqrt{n}
\end{equation}
for all $n \ge 2$. It turns out that the precise value of $g(n)$ has a nontrivial dependence on the arithmetic properties of $n$; indeed, the lower bound in~\eqref{eq:containment} is sharp for some positive integers, but strict for others. We record these facts, as well as some other properties of $g(\cdot)$, below. Since these observations don't seem to be enough resolve Question~\ref{qn:threshold} completely, we chose not to include detailed proofs of the claims below.

\begin{enumerate}[label = {\bfseries{G\arabic{enumi}}}]

\item Observe that if $d \geq 2$ and there exists a transitive projective plane of order $d$, then writing $n = d^2 + d +1$, we have $s(n,k) > 0$ if and only if $k \geq d+1$. Indeed, if $k \leq d$, then $s(n,k) = 0$ by Proposition~\ref{rootn}, while if $k \geq d+1$, then we start with a transitive projective plane $\mathbb{P}$ of order $d$ and take the family of all $k$-element subsets of the points of $\mathbb{P}$ containing a line of $\mathbb{P}$ to see that $s(n , k) >0$ in this case. In particular, for any odd prime power $q$, we have $s(q^2+q+1,k) > 0$ if and only if $k \geq q+1$; it follows that the lower bound in~\eqref{eq:containment} is sharp for any $n = q^2+q+1$ with $q$ an odd prime power, and consequently, we also get an affirmative answer to Question~\ref{qn:threshold} for all positive integers of this form.
\item On the other hand, the lower bound in~\eqref{eq:containment} is not tight for $n=43$, for example. It was shown by Lov\'asz~\citep{lov_pp} and F\"uredi~\citep{fur_pp} that if $d \ge 2$ is such that $n = d^2 + d + 1$ is prime, then $s(n, d + 1) > 0$ if and only if there exists a transitive projective plane of order $d$. Consequently, $s(43,7) = 0$ since $43$ is prime and there exists no projective plane of order $6$, so the lower bound in~\eqref{eq:containment} is not sharp in general. The Bateman--Horn conjecture~\citep{bh} would imply that $d^2+d+1$ is prime for infinitely many positive integers $d$ which are not themselves prime powers; taken together with the aforementioned observation of Lov\'asz and F\"uredi along with the non-existence conjecture for projective planes whose order is not a prime power, this would imply that $s(d^2+d+1,d+1) = 0$ for infinitely many $d \in \N$, and consequently that the lower bound in~\eqref{eq:containment} is not sharp for infinitely many positive integers.
\end{enumerate}

We can use other finite geometries in the place of projective planes to bound $g(\cdot)$; this allows us to answer Question~\ref{qn:threshold} in the affirmative for various sequences of positive integers with suitable `arithmetic structure'. 

\begin{enumerate}[label = {\bfseries{G\arabic{enumi}}}, resume]

\item For any prime power $q$, by taking the dual affine plane $\mathbb{DA}^2(\mathbb{F}_q)$ over the finite field $\mathbb{F}_q$ and considering the family of lines of $\mathbb{DA}^2(\mathbb{F}_q)$, then writing $n = q^2 + q$, we have $s(n,k) >0$ if $k \geq q+1$; this yields an affirmative answer to Question~\ref{qn:threshold} for any $n \in \N$ of this form.
\end{enumerate}

These constructions based on projective planes and dual affine planes have natural analogues based upon higher-dimensional projective spaces and higher-dimensional dual affine spaces, enabling us to answer Question \ref{qn:threshold} affirmatively for some other infinite sequences of integers.

\begin{enumerate}[label = {\bfseries{G\arabic{enumi}}}, resume]
\item Fix $r \in \mathbb{N}$, let $q$ be an prime power and consider the $(2r)$-dimensional projective space $\mathbb{P}^{2r}(\mathbb{F}_q)$ over $\mathbb{F}_q$. Then the family of all $r$-dimensional projective subspaces of $\mathbb{P}^{2r}(\mathbb{F}_q)$ gives us an affirmative answer to Question \ref{qn:threshold} for all $n = (q^{2r+1}-1)/(q-1)$ with $q$ a prime power.
\item Next, fix  $r \in \mathbb{N}$, let $q$ be a prime power, and consider the $(2r)$-dimensional dual affine space $\mathbb{DA}^{2r}(\mathbb{F}_q)$ over $\mathbb{F}_q$. Then the family of all $r$-flats of $\mathbb{DA}^{2r}(\mathbb{F}_q)$ gives us an affirmative answer to Question \ref{qn:threshold} for all $n = q(q^{2r}-1)/(q-1)$ with $q$ a prime power.
\end{enumerate}

For completeness, let us also record the following fact.
\begin{enumerate}[label = {\bfseries{G\arabic{enumi}}}, resume]
\item The observation of Banakh and Gavrylkiv~\citep{banakh} mentioned earlier shows that $g(n) = (1+o(1))\sqrt{n}$ whenever $n = q^2 - 1$ for some prime power $q$, or $n = p^2 - p$ for some prime $p$. Consequently, we have an affirmative answer to Question~\ref{qn:threshold} for any $n \in \N$ of the aforementioned forms.
\end{enumerate}

Finally, we demonstrate using a tensor product construction that $\mathcal{S}$ is closed under taking pointwise products. For a set $x \subset [n]$, we define its \emph{characteristic vector} $\chi_x \in \{0,1\}^n$ by $(\chi_x)_i = 1$ if $i \in x$ and $(\chi_x)_i=0$ otherwise. Given two sets $x\subset[n]$ and $y \subset[m]$, we define their \emph{tensor product $x \otimes y$} to be the subset of $[nm]$ whose characteristic vector $\chi_{x \otimes y}$ is given by 
\[(\chi_{x \otimes y})_{(i-1)m + j} = (\chi_x)_{i}(\chi_y)_j\]
for all $i \in [n]$ and $j \in [m]$. For two families $\A \subset \mathcal{P}_{n}$ and $\B \subset \mathcal{P}_{m}$, we define their tensor product by
\[\A \otimes \B = \{x \otimes y : x\in \A, y \in \B\};\]
note that $\A \otimes \B \subset \mathcal{P}_{nm}$ and that $|\mathcal{A} \otimes \mathcal{B}| = |\mathcal{A}||\mathcal{B}|$. Now observe that if $\mathcal{A} \subset [n]^{(k)}$ and $\B \subset [m]^{(l)}$, then $\A \otimes \B \subset [nm]^{(kl)}$, and furthermore, if $\A$ and $\B$ are symmetric and intersecting, then so is $\A \otimes \B$. It follows that
\[s(nm,kl) \geq s(n,k) s(m,l)\]
for all  $k,l,m,n \in \mathbb{N}$, and in particular, if $(n,k),(m,l) \in \mathcal{S}$, then $(nm,kl) \in \mathcal{S}$. 

\begin{enumerate}[label = {\bfseries{G\arabic{enumi}}}, resume]
\item The above observation implies that $g(\cdot)$ is submultiplicative, i.e., we have \[g(nm) \leq g(n) g(m)\] for all $n,m \in \mathbb{N}$. This fact may be used to answer Question~\ref{qn:threshold} affirmatively for some additional sequences of positive integers; for example, we conlcude that the answer to Question~\ref{qn:threshold} is in the affirmative for all $n = (q_1^2+q_1+1)(q_2^2+q_2+1)$ with $q_1$ and $q_2$ both prime powers, and so on.
\end{enumerate}

\section{Conclusion}\label{s:conc}
A number of interesting open problems remain. Theorem~\ref{thm:main} and Lemma~\ref{l:lbgen} together determine the order of magnitude of $\log ({n \choose k}/s(n,k))$ when $k/n$ is bounded away from zero by a positive constant. The gap between our upper and lower bounds for $s(n,k)$ is somewhat worse for smaller $k$, and it would be of interest to improve Theorem~\ref{thm:main} in the regime where $k=o(n)$.

Determining $s(n,k)$ precisely for all $k\le n/2$ would appear to be a challenging problem. We conjecture that for any $\delta>0$, if $n$ is sufficiently large depending on $\delta$ and $(1+\delta)\sqrt{n} \log n \leq k \leq n/2$, then \[s(n,k) = |\F(n,k)|.\] 
Note that if $n$ is sufficiently large depending on $\delta$ and $(1+\delta)\sqrt{n} \log n \leq k \leq n/2$, then the family $\F(n,k)$ yields a larger symmetric intersecting family than any of the algebraic constructions in Section~\ref{s:smallk}.

Determining the asymptotic behaviour of $g(n)$ is another problem that merits further investigation. We have established various estimates in Section~\ref{s:smallk}, but even the fundamental question of deciding whether $g(n)/\sqrt{n}$ converges in the limit as $n \to \infty$ still remains open.

\section*{Acknowledgements}
The second author wishes to acknowledge support from ERC Advanced Grant 320924, NSF grant DMS-1300120 and BSF grant 2014290, and the third author was partially supported by NSF Grant DMS-1800521. We would like to thank two anonymous referees for their careful reading of the paper. Finally, we would like to thank Nathan Keller and Omri Marcus for pointing out an error in a previous version of the paper; this led to an adjustment of the statement of Theorem \ref{thm:main}.

\bibliographystyle{amsplain}
\bibliography{sym_uni_fam}

\end{document}